\documentclass[12pt,letterpaper]{amsart}
\usepackage{palatino, euler, epic, eepic, amssymb, xypic, floatflt, microtype, enumerate, stmaryrd}
\input xy
\xyoption{all}

 \newlength{\baseunit}               
 \newcount{\numlines}                
\usepackage{amsmath, amssymb, amscd, verbatim, xspace,amsthm}
\usepackage{latexsym, epsfig, color}

\newcommand\tensor{\otimes}
\newcommand\isom{\cong}

\newcommand\Spec{\operatorname{Spec}}

\newcommand\bq{\begin{equation}}
\newcommand\eq{\end{equation}}

\newtheorem{proposition}{Proposition}[section]
\newtheorem{theorem}[proposition]{Theorem}

\newtheorem{lemma}[proposition]{Lemma}

\theoremstyle{definition}

\theoremstyle{remark}

\usepackage{url}

\numberwithin{equation}{section}

\newcommand{\cut}[1]{}

\newcommand\hidden[1]{}

\newcommand{\cM}{\mathcal{M}}

\renewcommand{\cR}{\mathcal{R}}
\newcommand{\cZ}{\mathcal{Z}}

\newcommand{\PP}{\mathbb{P}}
\newcommand{\QQ}{\mathbb{Q}}

\newcommand{\setmin}{{\smallsetminus}}                                %
\newcommand{\Mp}{{\mathcal{M}^+}}                                     %
\newcommand{\Rel}{{\mathcal{R}}}                                      %
\newcommand{\lar}{\longleftarrow}                                     %
\newcommand{\dra}{\dashrightarrow}                                    %
\newcommand{\ZZ}{{\mathbb{Z}}}                                        %
\newcommand{\LL}{{\mathbb{L}}}                                        %
\newcommand{\cO}{{\mathcal O}}                                        %
\newcommand{\ch}{\operatorname{\tilde{c}_1}}                          %
\newcommand{\Schk}{\operatorname{Sch}_{k}}                            %
\newcommand{\id}{\operatorname{Id}}                                   %
\title{Projectivity in Algebraic Cobordism}

\author{Jos\'e Luis Gonz\'alez and Kalle Karu}
\address{J.L. Gonz\'alez,  Dept. of Mathematics, University of British Columbia,
  Vancouver, BC V6T1Z2, CANADA  \newline \indent
K. Karu,
Dept. of Mathematics, University of British Columbia, 
  Vancouver, BC V6T1Z2, CANADA} 
\email{jgonza@math.ubc.ca, karu@math.ubc.ca}
\thanks{This research was funded by NSERC Discovery and Accelerator grants.}

\begin{document}
\begin{abstract}
The algebraic cobordism group of a scheme is generated by cycles that are proper morphisms from smooth quasiprojective varieties. We prove that over a field of characteristic zero the quasiprojectivity assumption can be omitted to get the same theory. 

\end{abstract}
\maketitle
\setcounter{tocdepth}{1} 




\section{Introduction}

The purpose of this article is to remove the quasiprojectivity assumption in the definition of the algebraic cobordism theory of Levine and Morel \cite{Levine-Morel}. Recall that the cobordism group $\Omega_*(X)$ of a scheme $X$ is generated by cycles of the form 
\[ [f:Y\to X, L_1,\ldots,L_r],\]
where $Y$ is a smooth quasiprojective variety, $f$ is a proper morphism, and $L_1,\ldots,L_r$ are line bundles on $Y$. We will construct a similar theory $\hat{\Omega}_*(X)$ in which the varieties $Y$ are assumed to be smooth, but not necessarily quasiprojective, and we prove that the natural morphism $\Omega_*(X) \to \hat{\Omega}_*(X)$ is an isomorphism.

Levine and Pandharipande in \cite{Levine-Pandharipande} gave a different definition of an algebraic cobordism theory $\omega_*(X)$ and proved it to be isomorphic to $\Omega_*(X)$. We show that in the definition of  $\omega_*(X)$ one can similarly remove the assumption that the varieties $Y$ are quasiprojective. Since the definition of $\omega_*(X)$ is simpler than $\Omega_*(X)$, we recall it here and we also explain the definition of the theory $\hat{\omega}_*(X)$. The definitions of $\Omega_*(X)$ and $\hat{\Omega}_*(X)$ are given in Section~\ref{sec-Omega}.

We work in the category $\Schk$ of separated finite type schemes over a field $k$ of characteristic zero. 
For $X$ in $\Schk$, let $\cM(X)$ be the set of isomorphism classes of proper morphisms $f: Y\to X$ where $Y$ is a smooth quasiprojective variety in $\Schk$. Let $\Mp(X)$ be the free abelian group with basis $\cM(X)$. The elements of $\Mp(X)$ are called cycles. The class of $f: Y\to X$ in $\Mp(X)$ is denoted by $[f: Y\to X]$. When $Y\to X$ is a morphism from a possibly reducible smooth scheme $Y$ in $\Schk$, then $[Y\to X]$ stands for the sum of cycles from the irreducible components of $Y$.

A double point degeneration is a morphism $\pi: Y\to \PP^1$, with $Y$ a smooth scheme of pure dimension, such that $Y_\infty = \pi^{-1}(\infty)$ is a smooth divisor on $Y$ and $Y_0=\pi^{-1}(0)$ is a union $A\cup B$ of smooth divisors intersecting transversely along $D=A\cap B$.  Define $\PP_D = \PP(\cO_Y(A)|_D\oplus \cO_D)$.  We say that the double point degeneration is quasiprojective if $Y$ is quasiprojective.

Let $X\in \Schk$ and let $Y$ be a smooth variety. Let $p_1, p_2$ be the two projections of $X\times \PP^1$.  A double point relation is defined by a proper morphism $\pi: Y\to X\times\PP^1$, such that $p_2\circ \pi: Y\to \PP^1$ is a quasiprojective double point degeneration. Let 
\[ [Y_\infty \to X], \quad [A\to X],\quad [B\to X], \quad [\PP_D \to X] \]
be the cycles obtained by composing with $p_1$. The double point relation is 
\[ [Y_\infty \to X] -[A\to X] - [B\to X] + [\PP_D\to X] \in \Mp(X). \]

Let $\Rel(X)$ be the subgroup of $\Mp(X)$ generated by all the double point relations. The cobordism group of $X$ is defined to be
\[ \omega_*(X) = \Mp(X)/\Rel(X).\]
The group $\Mp(X)$ is graded so that $[f: Y\to X]$ lies in degree $\dim Y$. Since double point relations are homogeneous, this grading gives a grading on $\omega_*(X)$. We write $\omega_n(X)$ for the degree $n$ part of  $\omega_*(X)$. 

Now let $\hat{\cM}(X)$ be the set of isomorphism classes of proper morphisms $Y\to X$, where $Y$ is a smooth variety in $\Schk$, and let  $\hat{\cM}^+(X)$ be the free abelian group generated by $\hat{\cM}(X)$. Let $\hat{\Rel}(X)$ be the subgroup of $\hat{\cM}^+(X)$ generated by double point relations as above, defined by proper morphisms $\pi: Y\to X\times\PP^1$, where $Y$ is smooth, but not necessarily quasiprojective. Define
\[ \hat\omega_*(X) = \hat{\cM}^+(X)/\hat\Rel(X).\]
There is a natural homomorphism $\psi: \omega_*(X) \to \hat\omega_*(X)$. The main result we prove is:

\begin{theorem}\label{thm-main} 
The homomorphism $\psi: \omega_*(X) \to \hat\omega_*(X)$ is an isomorphism for all $X$ in $\Schk$.
\end{theorem}

As a special case, consider $X=\Spec k$. It is proved in \cite{Levine-Morel, Levine-Pandharipande} that the vector space $\omega_*(\Spec k)\tensor \QQ$ is generated by the classes of products of projective spaces. In other words, for any smooth projective variety $Y$, the image of the class $[Y\to \Spec k]$ in $\omega_*(\Spec k)\tensor \QQ$ is equivalent, modulo double point relations, to a rational linear combination of products of projective spaces. The isomorphism $\hat{\omega}_*(\Spec k)\isom \omega_*(\Spec k)$ proves the same result for an arbitrary smooth complete $Y$.

The algebraic cobordism theory $\omega_*(X)$ has a
 functorial push-forward homomorphism $g_*: \omega_*(X)\to \omega_*(Z)$ for $g: X\to Z$ projective, and a functorial pull-back homomorphism $g^*: \omega_*(Z)\to \omega_{*+d}(X)$ for $g: X\to Z$ a  smooth quasiprojective morphism of relative dimension $d$. The two morphisms are both defined on the cycle level.  The cobordism theory also has exterior products $ \omega_*(X)\times \omega_*(W) \to \omega_*(X\times W)$, defined by 
 \[ [Y\to X] \times [Z\to W] = [Y\times Z \to X\times W].\]
These homomorphisms are compatible with first Chern class operators (see Section~\ref{sec-chernclass} below). In \cite{Levine-Morel}, a theory having projective push-forwards, smooth quasiprojective pull-backs, exterior products and first Chern class operators satisfying a set of compatibility conditions is called an oriented Borel-Moore functor with products. The theory $\omega_*(X)$ is such a theory.

In the proof of Theorem~\ref{thm-main} we will need to use the push-forward homomorphisms $g_*: \omega_*(X)\to \omega_*(Z)$ for $g: X\to Z$ proper, not necessarily projective. This push-forward homomorphism is defined at the cycle level the same way as the projective push-forward. One can easily check that it is functorial.

In the theory $\hat{\omega}_*(X)$ one can define push-forward along a proper morphism $g$, pull-back along a smooth (not necessarily quasiprojective) morphism $g$ and exterior products. These homomorphisms are defined at the cycle level, hence their functoriality and various compatibility conditions are easy to check. There is no straight-forward way to define the first Chern class operators on $\hat\omega_*(X)$. The isomorphism $\psi$ of Theorem~\ref{thm-main} induces first Chern class operators on  $\hat\omega_*(X)$ from those on $\omega_*(X)$. We will show below that these induced first Chern class operators on $\hat\omega_*(X)$ satisfy the expected properties, for example the section axiom. It is then elementary to check that the first Chern class operators are compatible with proper push-forwards, smooth pull-backs and exterior products. In summary, one can strengthen the notion of the Borel-Moore functor with products by requiring the existence of proper push-forward and smooth pull-back homomorphisms,
 with 
the same compatibility conditions as in \cite{Levine-Morel}. Then $\hat{\omega}_*$ is such a functor.

The exterior products turn $\omega_*(\Spec k)$ into a graded ring and $\omega_*(X)$ into a graded module over $\omega_*(\Spec k)$. When $X$ is a smooth quasiprojective variety, we denote by $1_X$ the class $[id_X: X\to X] \in \omega_*(X)$. Similarly, $\hat\omega_*(\Spec k)$ is a ring and  $\hat\omega_*(X)$ is a module over this ring. For a smooth variety $X$, we have the class $1_X= [id_X: X\to X] \in \hat\omega_*(X)$.

The proof of Theorem~\ref{thm-main} is very similar to the proof of the main result in \cite{descentseq}. The main difference is that the theory $\hat{\omega}_*(X)$  does not have first Chern class operators. In the proof we need to be careful that first Chern class operators are applied in the theory $\omega_*(X)$ only.


\section{First Chern classes and Divisor classes}          \label{sec-chernclass}

We recall here the formal group law, the first Chern class operators and the divisor classes in the theory $\omega_*(X)$. We will use the notation $\omega_*(X)$ for algebraic cobordism. Since $\omega_*(X)\isom \Omega_*(X)$, the same holds for  $\Omega_*(X)$.

\subsection{Formal group law}

A formal group law on a commutative ring $R$ is a power series $F_R(u,v)\in R\llbracket u,v\rrbracket $ satisfying
\begin{enumerate}[(a)]
\item $F_R(u,0) = F_R(0,u) = u$,
\item $F_R(u,v)=  F_R(v,u)$,
\item $F_R(F_R(u,v),w) = F_R(u,F_R(v,w))$.
\end{enumerate}
Thus 
\[ F_R(u,v) = u+v +\sum_{i,j>0} a_{i,j} u^i v^j,\]
where $a_{i,j}\in R$ satisfy $a_{i,j}=a_{j,i}$ and some additional relations coming from property (c). We think of $F_R$ as giving a formal addition
\[ u+_{F_R} v = F_R(u,v).\]
There exists a unique power series $\chi(u) \in R\llbracket u\rrbracket $ such that $F_R(u,\chi(u)) = 0$. Denote $[-1]_{F_R} u = \chi(u)$. Composing $F_R$ and $\chi$, we can form linear combinations 
\[ [n_1]_{F_R} u_1 +_{F_R}   [n_2]_{F_R} u_2 +_{F_R} \cdots +_{F_R} [n_r]_{F_R} u_r \in R\llbracket u_1,\ldots,u_r\rrbracket \]
for $n_i\in \ZZ$ and $u_i$ variables.

There exists a universal formal group law $F_\LL$, and its coefficient ring $\LL$ is called the \emph{Lazard ring}. This ring can be constructed as the quotient of the polynomial ring $\ZZ[A_{i,j}]_{i,j>0}$ by the relations imposed by the three axioms above. The images of the variables $A_{i,j}$ in the quotient ring are the coefficients $a_{i,j}$ of the formal group law $F_\LL$. The ring $\LL$ is graded, with $A_{i,j}$ having degree $i+j-1$. The power series $F_\LL(u,v)$ is then homogeneous of degree $-1$ if $u$ and $v$ both have degree $-1$. We sometimes write $\LL=\LL_*$ to emphasize the grading on $\LL$.

It is shown in \cite{Levine-Morel,Levine-Pandharipande} that the ring $\omega_*(\Spec k)$ is isomorphic to $\LL$. The formal group law on $\LL$ describes the first Chern class operators of tensor products of line bundles (property (FGL)  below).

\subsection{First Chern Class Operators} \label{sect-chern}

Algebraic cobordism is endowed with first Chern class operators 
\[ \ch(L): \omega_*(X)\to \omega_{*-1}(X),\]
associated to a line bundle $L$ on $X$.  
We list three properties satisfied by the first Chern class operators $\ch(L): \omega_*(Y) \to \omega_{*-1}(Y)$ for $Y$  a smooth quasiprojective scheme and $L$ a line bundle on $Y$:
\begin{itemize}
\item[(Dim)] For $L_1,\ldots,L_r$ line bundles on $Y$, $r>\dim Y$,
\[\ch(L_1)\circ\cdots\circ \ch(L_r) (1_Y) = 0.\]
\item[(Sect)] If $L$ is a line bundle on $Y$ and $s\in H^0(Y,L)$ is a section such that the zero subscheme $i:Z\hookrightarrow Y$ of $s$ is smooth, then
\[ \ch(L)(1_Y) = i_*(1_Z).\]
\item[(FGL)] For two line bundles $L$ and $M$ on $Y$,
\[ \ch(L\otimes M)(1_Y) = F_\LL(\ch(L), \ch(M))(1_Y).\]
\end{itemize}

In the terminology of \cite{Levine-Morel, Levine-Pandharipande} the three properties imply that $\omega_*$ is an oriented Borel-Moore functor of geometric type. 

The first Chern class operators of two line bundles commute: $\ch(L)\circ\ch(M) = \ch(M)\circ \ch(L)$, and they satisfy a set of compatibility relations with smooth quasiprojective pull-backs, projective push-forwards and exterior products (see \cite{Levine-Morel}). As an example of these relations, for $L$ a line bundle on $X$ we have
\[ \ch(L)[f:Y\to X] = f_* \ch(f^*(L))(1_Y).\]
The property (Sect) above implies that if $L$ is a trivial line bundle on $X$, then the first Chern class operator of $L$ is zero.

\subsection{Divisor Classes} \label{sec-div-cl}

Recall that a divisor $D$ on a smooth scheme $Y \in \Schk$ has strict normal crossings (s.n.c.) if at every point $p\in Y$ there exists a system of regular parameters $y_1,\ldots,y_n$, such that  $D$ is defined by the equation $y_1^{m_1}\cdots y_n^{m_n}=0$ near $p$ for some integers $m_1, \ldots, m_n$. We let $|D|$ denote the support of $D$.

Let $D = \sum_{i=1}^r n_i D_i$ be a nonzero s.n.c. divisor on a smooth scheme $Y$, with $D_i$ irreducible. Let us recall the construction by Levine and Morel \cite{Levine-Morel} of the class $[D\to |D|] \in \omega_*(|D|)$. We do not need to assume that $Y$ is quasiprojective, however, we need that each component of $D$ is quasiprojective.

Let 
\[ F^{n_1,\ldots,n_r}(u_1,\ldots,u_r) = [n_1]_{F_\LL} u_1 +_{F_\LL}   [n_2]_{F_\LL} u_2 +_{F_\LL} \cdots +_{F_\LL} [n_r]_{F_\LL} u_r \in \LL\llbracket u_1,\ldots,u_r\rrbracket .\]
We decompose this power series as 
\[ F^{n_1,\ldots,n_r}(u_1,\ldots,u_r) = \sum_J F_J^{n_1,\ldots,n_r}(u_1,\ldots,u_r) \prod_{i\in J} u_i,\]
where the sum runs over nonempty subsets $J\subset \{1,\ldots, r\}$. The power series $F_J^{n_1,\ldots,n_r}$ are such that $u_i$ does not divide any nonzero term in $F_J^{n_1,\ldots,n_r}$ if $i\notin J$. 

For $i=1,\ldots,r$, let $L_i=\cO_Y(D_i)$. If $J\subset \{1,\ldots,r\}$, let  $i^J: D^J=\cap_{i\in J} D_i \hookrightarrow |D|$, and $L_i^J =  L_i|_{D^J}$. The class $[D\to|D|]$ is defined in \cite{Levine-Morel} as
\begin{equation} \label{eq2}
  [D\to|D|] = \sum_J i_*^J F_J^{n_1,\ldots,n_r} (L_1^J, \ldots, L_r^J) (1_{D^J}),
\end{equation}
where the sum runs over nonempty subsets $J\subset \{1,\ldots,r\}$ and   
$F_J^{n_1,\ldots,n_r} (L_1^J, \ldots, L_r^J)$ is the power series $F_J^{n_1,\ldots,n_r}$ evaluated on the first Chern classes of $L_1^J, \ldots, L_r^J$. 

We note that in the definition of divisor classes it is not necessary to assume that $D_i$ are irreducible. We may let them be smooth but possibly reducible divisors and then the same formula holds.

When $Y$ is quasiprojective, then the class $[D\to|D|]$ pushed forward to $Y$ becomes equal to $\ch(\cO(D))(1_Y)$. This gives a way to construct first Chern class operators in the theory $\omega_*$. If $L$ is a line bundle on $Y$, write $L=\cO_Y(A-B)$ for smooth divisors $A$ and $B$ that intersect transversely. The divisor class $[A-B\to |A-B|]$, when pushed forward to $Y$ then equals $\ch(L)(1_Y)$. The divisors $A$ and $B$ can be chosen with the help of an ample line bundle, but such divisors may not exist when $Y$ is not quasiprojective.

\subsection{Product of Divisor Classes}

Let  $D$ and $E$ be divisors on a smooth scheme $W$, such that $|D|\cup|E|$ has s.n.c. We recall from \cite{descentseq} the construction of the class 
\[ [D\bullet E \to |D|\cap|E|] \in \omega_*( |D|\cap|E|),\]
such that when $W$ is quasiprojective, then the push-forward of this class to $W$ is equal to 
\[ \ch(\cO_W(D))\circ \ch(\cO_W(E)) (1_W).\]
To define the product class, it is enough to assume that one of the two divisors has all its components quasiprojective. In the discussion below we will assume the quasiprojectivity of the components of $E$ but not of the ambient space $W$.

Let $D=\sum_i n_i D_i$ and $E=\sum_i p_i D_i$, where $D_i$, $i=1,\ldots,r$ are irreducible divisors. For $i=1,\ldots,r$, let $L_i=\cO_W(D_i)$. If $J\subset \{1,\ldots,r\}$ is such that $n_j\neq 0$ and $p_i\neq 0$ for some $i,j\in J$,  let  $i^J: D^J=\cap_{i\in J} D_i \hookrightarrow |D|\cap |E|$, and $L_i^J =  L_i|_{D^J}$.

Let the class $[D\bullet E \to |D|\cap|E|]$ be defined by the formula 
\[ \sum_{I, J} i_*^{I\cup J} F_J^{n_1,\ldots,n_r} (L_1^{I\cup J}, \ldots, L_r^{I\cup J}) F_I^{p_1,\ldots,p_r} (L_1^{I\cup J}, \ldots, L_r^{I\cup J}) \prod_{i\in I\cap J} \ch(L_i^{I\cup J})  (1_{D^{I\cup J}}).\]
Here the sum runs over pairs of nonempty subsets $I$ and $J$ of $\{1,\ldots,r\}$, such that $n_j\neq 0$ and $p_i\neq 0$ for all $j\in J$ and $i\in I$.  As in the case of divisor classes, it is enough to assume that the divisors $D_i$ are smooth but not necessarily irreducible. 

The following properties of the product of divisor classes are proved in  \cite{descentseq}:
\begin{enumerate}
\item When $D$ is a smooth reduced divisor that does not have common components with $E$, then $E'=E|_{|D|}$ is an s.n.c. divisor on $|D|$ and we have
\[ [D\bullet E \to |D|\cap|E|] = [E' \to |E'|].\]
\item The class $[D\bullet E \to |D|\cap|E|]$, when pushed forward to $|E|$, becomes equal to  
\[\ch(\cO_W(D)|_{|E|}) [E\to |E|].\] 
\item The formula for the product of divisor classes is symmetric, but not linear in either argument. Indeed, if $D=D'+F$, then the product class pushed forward to $E$ is 
\begin{gather*}  \ch(\cO_W(D'+F)|_{|E|}) [E\to |E|] = \ch(\cO_W(D')|_{|E|}) [E\to |E|] +\ch(\cO_W(F)|_{|E|}) [E\to |E|] + \\ \sum_{i,j>0} a_{i,j} \ch(\cO_W(D')|_{|E|})^i \ch(\cO_W(F)|_{|E|})^j [E\to |E|].
\end{gather*}
Here $a_{i,j}$ are the coefficients of the formal group law. However, notice that if $\ch(\cO_W(F)|_{|E|}) [E\to |E|]=0$, then the last two summands vanish and the classes
\[  [D\bullet E \to |D|\cap|E|]  \quad \textnormal{and} \quad [D'\bullet E \to |D|\cap|E|] \]
become equal when pushed forward to $|E|$.
\item As a special case of the previous argument, given $f: W\to \PP^1\times\PP^1$, let $D=f^{*}(\PP^1\times\{0\})$, $E=f^{*}(\{0\}\times \PP^1)$. Assume that $D$ and $E$ satisfy the normal crossing assumption and $E$ has every component quasiprojective. Let  $D=D'+F$, where $F$ contains the components of $D$ that map to $(0,0)$ and $D'$ contains the components that map onto $\PP^1\times\{0\}$. Then the class $[F\bullet E \to |F|\cap|E|]$ becomes zero when pushed forward to $|E|$. Indeed, the product class becomes zero when pushed forward to $|F|$ because $E|_{|F|}$ is trivial, and the morphism to $E$ factors: $|E|\cap|F| \to |F|\to |E|$. It follows that  the classes
\[  [D\bullet E \to |D|\cap|E|], \qquad [D'\bullet E \to |D|\cap|E|] \]
are equal when pushed forward to $E$.
\end{enumerate}


\section{Chow's lemmas}

The classical Chow's lemma states that every variety $Y$ admits a proper birational morphism $f:Y'\to Y$ from a quasiprojective variety $Y'$. Since $Y'$ is quasiprojective, the morphism $f$ is projective and hence $Y'$ is the blowup of $Y$ along an ideal sheaf $I$ on $Y$.

Hironaka's version of Chow's lemma is as follows (see \cite{Hironaka}). Suppose $Y$ is a smooth variety with $D$ an s.n.c. divisor on $Y$. Then there exists a proper birational morphism $g:\tilde{Y}\to Y$ from a quasiprojective variety $\tilde{Y}$, obtained by a sequence of blowups along smooth centers that intersect the exceptional locus together with the inverse image of $D$ normally. One can construct such $\tilde{Y}$ as the principalization of the ideal sheaf $I$ from the classical Chow's lemma. Then $g: \tilde{Y} \to Y$ factors through $f: Y'\to Y$. Since $f$ and $g$ are projective, the morphism $\tilde{Y}\to Y'$ is also projective, hence $\tilde{Y}$ is quasiprojective. 

Recall that if $Y$ is smooth and $D$ is an s.n.c. divisor on $Y$, then a
smooth subscheme $C\subset Y$ is said to intersect $D$ normally if at
every point $p\in Y$ we can choose a regular system of parameters
$y_1,\ldots, y_r$ so that $D$ is defined by $y_1^{n_1}\cdots
y_r^{n_r}=0$ for some $n_1,\ldots,n_r\in \ZZ$ and $C$ is defined by
vanishing of $y_{i_1},\ldots,y_{i_j}$ for some $i_1,\ldots,i_j$. If $D$
is an s.n.c. divisor and $C$ intersects it normally, then the blowup of
$Y$ along $C$ is smooth and the pull-back of $D$ together with the
exceptional divisor is again an s.n.c. divisor.

We refine Hironaka's version of Chow's lemma further as follows. Define an invariant $\nu(Y)$ of a variety $Y$ as the minimum dimension of a variety $Z$, such that there exists a projective morphism $ Y\to Z$. This is well-defined because $\id_Y:Y\to Y$ is projective. We have that $\nu(Y)=0$ if and only if $Y$ is projective. Consider now $\pi: Y\to Z$ with $\dim(Z)=\nu(Y)$. By classical Chow's lemma there exists a quasiprojective $Z'$ that is a blowup of $Z$ along an ideal sheaf $I$ on $Z$. Let $\tilde{Y}\to Y$ be the principalization of the inverse image ideal sheaf $\pi^{-1}(I) \cdot \mathcal{O}_Y$ on $Y$. Then $\tilde{Y}$ is quasiprojective by the same argument as before. Moreover, the centers of the blowups in the principalization all lie over the co-support of $I$, hence the centers have their $\nu$ invariant strictly smaller than $\nu(Y)$. 

\begin{lemma}[Embedded Chow's lemma] \label{lem-emb-Chow}
Let $W$ be a smooth variety, $D, E$ be
effective divisors on $W$ such that $D+E$ has s.n.c. Then there exists a birational morphism $g:\tilde{W}\to
W$, obtained by a sequence of blowups of smooth centers that lie over
$|D|$ and intersect the pull-back of $D+E$ together with the exceptional
locus normally, such that every component $\tilde{D}_i$ of the
pull-back $\tilde{D}=g^*(D)$ is quasiprojective.
\end{lemma}

\begin{proof}
Let $D_i$ be a component of $D$. Let $D'$ be the divisor
$D'=(D+E-mD_i)|_{D_i}$, where $m$ is the coefficient of $D_i$ in $D+E$.
We apply the refined version of Chow's lemma to the variety $D_i$ and divisor $D'|_{D_i}$. This gives a quasiprojective variety $\tilde{D}_i$ obtained by a sequence of blowups of $D_i$ along smooth centers. Let us now perform the same sequence of blowups on $W$ (blow up $W$
along the same centers lying in $D_i$ and in the strict transforms of
$D_i$), to get $g: \tilde{W}\to W$. Then $\tilde{D}_i$ is isomorphic to
the strict transform of $D_i$ in $\tilde{W}$. Such blowups introduce new
components to the divisor $g^*(D)$. However, the new exceptional divisors are projective bundles over the centers of blowups, hence their $\nu$-invariant is strictly smaller than $\nu(D_i)$. By induction on $\nu$ we can make all components quasiprojective.
\end{proof}

Recall the natural homomorphism $\psi: \omega(X)\to \hat{\omega}(X)$ from Theorem~\ref{thm-main}.

\begin{lemma}\label{lem-surj} The homomorphism
 \[ \psi: \omega(X)\to \hat{\omega}(X).\]
is surjective.
\end{lemma}

\begin{proof}
Let $[Y\to X]$ be a cycle in $\hat{\cM}(X)$. Consider a sequence of blowups along smooth centers
\[ Y=Y_0\lar Y_1\lar \cdots \lar Y_n,\]
such that $Y_n$ is quasiprojective. 

Each blowup $Y_{i+1}\to Y_i$ along a center $C_i$ gives a double point relation in  $\hat{\omega}(X)$. Let $W= Y_i\times \PP^1$, and let $\tilde{W}\to W$ be the blowup along $C_i\times\{0\}\subset W$. Then $\tilde{W}_0$ has two components, $Y_{i+1}$ and the exceptional divisor $E_i$, intersecting transversely along $D_i$. The double point relation in $\hat{\omega}(X)$ is
\[ [Y_i\to X] = [Y_{i+1}\to X] + [E_i\to X] - [\PP_{D_i}\to X].\]
Combining these for all $i$, we get
\[ [Y\to X] = [Y_n\to X] + \sum_i \left( [E_i\to X] - [\PP_{D_i}\to X] \right) .\]
The class $[Y_n\to X]$ lies in the image of $\psi$ because $Y_n$ is quasiprojective. Since $E_i$ and $\PP_{D_i}$ are projective over the center $C_i$, their $\nu$-invariant is strictly smaller than $\nu(Y)$ and we may assume by induction on $\nu$ that the classes $[E_i\to X]$ and  $[\PP_{D_i}\to X]$ also lie in the image of $\psi$.
\end{proof}


\section{Proof of the main theorem}

We will now prove Theorem~\ref{thm-main}. We follow the argument of the main result in \cite{descentseq}, which in its turn was modeled after a proof in  \cite{Levine-Morel}. The proof that $\psi$ is injective has two steps:
\begin{enumerate}
 \item Define a distinguished lifting $\hat{\cM}^+(X)\stackrel{d}{\longrightarrow} \omega_*(X)$, such that the composition 
 \[ {\Mp}(X) {\longrightarrow} \hat{\cM}^+(X)  \stackrel{d}{\longrightarrow} \omega_*(X) \]
 is the canonical homomorphism.
\item Show that $d$ maps $\hat{\Rel}(X)$ to zero, hence it descends to $d: \hat{\omega}_*(X) \to \omega_*(X)$, providing a left inverse to $\psi$ and proving that $\psi$ is injective.
\end{enumerate}

Once we prove that $\psi$ is injective, Lemma~\ref{lem-surj} implies that $\psi$ is an isomorphism, hence the left inverse $d$ of $\psi$ is in fact a two-sided inverse.

\subsection{Distinguished Liftings}

Let $[Y\to X]$ be a cycle in $\hat{\cM}(X)$. We construct a set of elements in $\omega_*(X)$ that we call distinguished liftings of the cycle.

Let $W$ be a smooth variety and $\pi:W\to Y\times \PP^1$ a proper birational morphism that is an isomorphism over $Y\times(\PP^1\setmin\{0\})$.   Assume that
 $D=\pi^*(Y\times \{0\}) $ is a divisor of s.n.c on $W$ with every component quasiprojective. Given such $W$, the class $[D\to |D|]$ pushed forward to $X$ is called a distinguished lifting of the cycle $[Y\to X]$. 

Note that, even though every component of $D$ is quasiprojective, the scheme $|D|$ may fail to be quasiprojective and the map $|D|\to X$ may not be projective. Thus the push-forward homomorphism $\omega_*(|D|)\to \omega_*(X)$ is along a proper morphism. It follows from the functoriality of the push-forward homomorphism that the distinguished liftings of $[Y\to X]$ are the distinguished liftings of $1_Y$ pushed forward to $X$.

Distinguished liftings always exist. For example, by Lemma~\ref{lem-emb-Chow}, we can obtain a variety $W$ by a sequence of blowups of $Y\times\PP^1$ with centers lying over $Y\times \{0\}$. An arbitrary $W$ as described above is obtained from $Y\times\PP^1$ by a sequence of blowups and blowdowns along smooth centers lying over $Y\times \{0\}$ \cite{Wlodarczyk, AKMW}.

\begin{lemma}
Any two distinguished liftings of a cycle $[Y\to X] \in \hat\cM(X)$ are equal in $\omega_*(X)$.
\end{lemma}

\begin{proof}
Consider smooth varieties $W_1$ and $W_2$ defining two distinguished liftings as images of  $[D_1\to |D_1|]$ and $[D_2\to |D_2|]$ in $\omega_*(X)$. 
We will prove below that if the birational map $h: W_1\dra W_2$ is a projective morphism, then the class $[D_1\to |D_1|]$, when pushed forward to $|D_2|$, becomes equal  to $[D_2\to |D_2|]$. This clearly proves the lemma in case $h$ is a projective morphism. (Notice that we need to use the functoriality of the push-forward homomorphism along the composition $|D_1|\to |D_2| \to X$. The morphisms here are proper, but not necessarily projective.) The case of general $h$ can be reduced to the case of projective $h$ as follows. There exist projective morphisms $W_1'\to W_1$ and $W_2'\to W_2$, both obtained by sequences of blowups along smooth centers lying over $Y\times\{0\} \in \PP^1$, such that the birational map 
$W_1'\dra W_2'$ is a projective morphism. (See Lemma 1.3.1 in \cite{AKMW} for a proof.) We may then apply the case of projective $h$ to each one of the three projective morphisms.

Now assume that $h:W_1\to W_2$ is a projective morphism. By the weak factorization theorem \cite{Wlodarczyk, AKMW}, we can factor $h$ into a sequence of blowups and blowdowns along smooth centers. Moreover, the factorization can be chosen so that if $Z_{i+1}\to Z_{i}$ is one blowup of $C\subset Z_i$ in this factorization,  then the birational map $g_i: Z_i \dra W_2$ is a projective morphism, $g_i^*(D_2)$ is an s.n.c. divisor on $Z_i$, the center $C$ lies in the support of the divisor $g_i^*(D_2)$ and intersects it normally. 

We may thus assume that $h: W_1 \to W_2$ is the blowup of $W_2$ along a smooth center $C\subset W_2$ that lies in the support of $D_2$ and intersects it normally. 

Let $V_2 = W_2\times\PP^1$. Let  $V_1$ be the blowup of $V_2$ along $C\times\{0\} \subset V_2$. Let $f: V_1 \to \PP^1\times\PP^1$  be the projection. Consider the divisors $D=f^*(\PP^1\times\{0\})$ and $E=f^*(\{0\}\times \PP^1)$. Then $D+E$ is an s.n.c. divisor  and $E$ has all its components quasiprojective. Moreover, $D=D'+F$, where $F$ is the exceptional divisor of the blowup, lying over $(0,0)\in\PP^1\times\PP^1$, and $D'\isom W_1$. Since $D'$ is smooth, having no common component with $E$, and $E|_{D'} = D_1$, we get 
\[ [D'\bullet E\to |D'|\cap |E|] = [D_1\to |D_1|]. \]
When pushed forward to $|E|$, this class becomes equal to $\ch(\cO(D))[E\to |E|]$, which itself is equal to  $[D_2\to |D_2|]$ pushed forward to $|E|$ by a section $s: |D_2|\to |E|$ of the projection $|E|\to  |D_2|$ . The two classes are equal when pushed forward to $|D_2|$.
\end{proof}

The previous lemma proves that distinguished liftings are unique. 
We extend the liftings of generators $[Y\to X]$ linearly to a group homomorphism $d: \hat{\cM}^+(X)\to \omega_*(X)$. If $Y$ is quasiprojective, then we can take the distinguished lifting of $[Y\to X]\in\hat\omega_*(X)$ to be $[Y\to X]\in \omega_*(X)$,  hence the composition
 \[ \Mp(X) {\longrightarrow} \hat{\cM}^+(X)  \stackrel{d}{\longrightarrow} \omega_*(X) \]
is the canonical projection.

\subsection{Proof of Theorem~\ref{thm-main}}

It remains to prove that the distinguished lifting $d: \hat{\cM}^+(X)\to \omega_*(X)$ maps $\hat{\Rel}(X)$ to zero. Consider a double point degeneration $f:W\to \PP^1$, $W\to X$.  Let $W_\infty = f^{-1}(\infty)$ be a smooth fiber and  $W_0 = f^{-1}(0) = A \cup B$. Recall that the double point relation is
\[ [W_\infty\to X] - [A\to X] - [B\to X] +[\PP_{A\cap B}\to X],\]
where $\PP_{A\cap B}=  \PP(\cO_{A\cap B}(A) \oplus \cO_{A\cap B})$. Since $\Rel(X)$ is generated by the double point relations, it suffices to prove that 
\[  d[W_\infty\to X] - d[A\to X] -d [B\to X] +d[\PP_{A\cap B}\to X] = 0 .\]

Let $V=W\times \PP^1$. We blow up $V$ along smooth centers lying over $W\times \{0\}$ that intersect the pull-back of $W_0\times \PP^1+ W_\infty\times \PP^1+ W\times\{0\}$  normally.  Let the result be $\tilde{V}$, such that the inverse image of $W\times \{0\}$ has every component quasiprojective.  Let
\[ g: \tilde{V} \to W\times\PP^1 \to \PP^1\times\PP^1.\]
Define 
\[ E=g^*(\PP^1\times \{0\}),  \quad D_0=g^*(\{0\}\times \PP^1), \quad  D_\infty=g^*(\{\infty\}\times \PP^1).\]
By construction, all components of $E$ are quasiprojective.
Let $D_0'$, $D_\infty'$ be the sums of components in $D_0$, $D_\infty$ that do not map to $(0,0)$ or $(\infty,0)$ in $\PP^1\times\PP^1$. Then $D_\infty'$ is the blowup of $W_\infty\times\PP^1$ along centers lying over $W_\infty\times \{0\}$.  Since $D'_\infty$ is smooth and has no component in common with $E$, it follows that  $[D_\infty'\bullet E \to |D_\infty'|\cap |E|]$ is the divisor class of $E|_{D'_\infty}$, which  gives the distinguished lifting of $[W_\infty\to X]$. 
 
 Similarly,  $D_0'$ is the blowup of $W_0\times\PP^1 = (A\cup B)\times \PP^1$ along centers lying over $W_0\times \{0\}$. The divisor $D_0'$ is a union $A'\cup B'$ of  two smooth divisors, the blowups of $A\times \PP^1$ and $B\times \PP^1$. The intersection of these divisors is a blowup of $(A\cap B)\times\PP^1$. It is proved in \cite{descentseq}  that in this situation $[D_0'\bullet E \to |D_0'|\cap |E|]$, gives the class $d[A\to X] +d [B\to X] -d[\PP_{A\cap B}\to X]$. 
 
 When pushed forward to $|E|$, the classes $[D_\infty'\bullet E \to |D_\infty'|\cap |E|]$ and $[D_0'\bullet E \to |D_0'|\cap |E|]$ are equal to $\ch(\cO_{|E|}(D_\infty))[E\to |E|]$ and $\ch(\cO_{|E|}(D_0))[E\to |E|]$. Since $D_0$ and $D_\infty$ are linearly equivalent divisors, the two classes become equal in $\omega_*(|E|)$, hence their push-forwards are equal in $\omega_*(X)$.
\qed

\subsection{First Chern class operators on $\hat{\omega}_*(X)$.}

The isomorphism $\psi:\omega_*(X) \to \hat{\omega}_*(X)$ induces first Chern class operators on $\hat{\omega}_*(X)$. We will show below that the property (Sect) holds in $\hat{\omega}_*(X)$ (recall the definitions of properties (Dim), (Sect) and (FGL) in Section~\ref{sect-chern}). The other two properties (Dim) and (FGL) follow trivially from the same properties in $\omega_*(X)$.

\begin{lemma} \label{lem-sect} 
 Let $Y$ be a smooth variety, $L$ a line bundle on $Y$, and $i: Z\to Y$ the closed embedding of the subscheme defined by a transverse section of $L$. Then in $\hat\omega_*(Y)$
\[ \ch(L) (1_Y) = [i: Z\to Y].\]
\end{lemma}

\begin{proof}
We need to show that in $\omega_*(Y)$
 \[ \ch(L) (d(1_Y)) = d [i: Z\to Y]. \]

To construct $d(1_Y)$, let $f: W\to Y\times \PP^1$ be the blowup along smooth centers lying over $Y\times\{0\}$, such that $E=f^*(Y\times \{0\})$ has all its components quasiprojective. Let $D = f^*(Z\times \PP^1)$. We may also assume that $E+D$ has s.n.c. 

Write $D=D'+F$, where $D'$ is the strict transform of $Z\times \PP^1$. Then the class $[E\bullet F\to |E|\cap|F|]$, when pushed forward to $|F|$ is equal to 
\[ \ch(\cO(E))([F\to |F|]) = 0\]
because $\cO(E)$ is trivial on $F$. (Note that $|F|\subset |E|$, hence $F$ also has all its components quasiprojective and the divisor class $[F\to |F|]$ makes sense.) The push-forward of the class $[E\bullet F\to |E|\cap|F|]$ to $|E|$ factors through $|F|$, hence it is zero. This implies that, when pushed forward to $|E|$, the classes
\[ [E\bullet D\to |E|\cap|D|],\quad [E\bullet D'\to |E|\cap|D'|] \]
become equal. The first class pushed forward to $E$ is 
 \[ \ch(\cO(D))([E\to |E|]),\]
which further pushed forward to $Y$ becomes equal to $\ch(L)(1_Y)$. Since $D'$ is smooth and has no components in common with $E$, the second class is equal to $[E|_{|D'|}\to |E|_{|D'|}|]$, which becomes equal to $d [i: Z\to Y]$ when pushed forward to $Y$.
\end{proof}

Recall that we defined the divisor class $[D\to |D|] \in \omega_*(|D|)$ for $D$ an s.n.c. divisor on a smooth variety $Y$, assuming that all components of $D$ are quasiprojective. Given first Chern class operators on $\hat{\omega}_*$, the same formula for $[D\to |D|] \in \hat{\omega}_*(|D|)$ makes sense even when the components of $D$ are not quasiprojective. The previous lemma then implies that The class 
$[D\to |D|]$ pushed forward to $Y$ becomes equal to  $\ch(\cO_Y(D))(1_Y) \in \hat{\omega}_*(Y)$, with no assumptions about quasiprojectivity of $Y$ or $D$.


\section{The theories $\Omega_*(X)$ and $\hat\Omega_*(X)$}         \label{sec-Omega} 

We start by recalling the construction of the algebraic cobordism theory $\Omega_*(X)$ in \cite{Levine-Morel}.

Let $\cZ_*(X)$ be the graded free abelian group generated by isomorphism classes of cobordism cycles of the form
\[ [f:Y\to X, L_1,\ldots,L_r],\]
where $Y$ is a smooth quasiprojective variety, $f$ is a proper morphism, and $L_i$ are line bundles on $Y$. An isomorphism class means an isomorphism  class of $f$ together with isomorphism classes of line bundles $L_i$, possibly after a permutation. The degree of the cycle above is $\dim Y -r$.  
The groups $\cZ_*(X)$ have a
 functorial push-forward homomorphism $g_*: \cZ_*(X)\to \cZ_*(Z)$ for $g: X\to Z$ proper, and a functorial pull-back homomorphism $g^*: \cZ_*(Z)\to \cZ_{*+d}(X)$ for $g: X\to Z$ a  smooth quasiprojective morphism of relative dimension $d$. The first Chern class operators on $\cZ_*(X)$ are defined by 
\[ \ch(L) [f:Y\to X, L_1,\ldots,L_r] = [f:Y\to X, L_1,\ldots,L_r, f^*(L)]\]
for $L$ a line bundle on $X$.

The cobordism groups $\Omega_*(X)$ are defined by imposing relations on $\cZ_*(X)$.
The construction of $\Omega_*(X)$ is given in three steps corresponding to properties (Dim), (Sect) and (FGL).

\begin{enumerate}
 \item Let $\langle\cR_*^{\text{Dim}}\rangle(X)$ be the subgroup of $\cZ_*(X)$ generated by cobordism cycles of the form 
\[ [f:Y\to X, \pi^*(L_1),\ldots,\pi^*(L_r), M_1,\ldots, M_s],\]
where $Z$ is a smooth quasiprojective variety, $\pi: Y\to Z$ is a smooth quasiprojective morphism, $L_1,\ldots,L_r$ are line bundles on $Z$, and $r>\dim Z$. Define
\[ \underline\cZ_*(X) = \cZ_*(X)/\langle\cR_*^{\text{Dim}}\rangle(X).\]

\item Let $\langle\cR_*^{\text{Sect}}\rangle(X)$ be the subgroup of $\underline\cZ_*(X)$ generated by differences of cobordism cycles of the form
\[ [Y\to X, L_1,\ldots,L_r] - [Z\to X, i^*(L_1),\ldots,i^*(L_{r-1})],\]
where $i: Z \to Y$ is the closed embedding of the zero locus of a transverse section of $L_r$. Define
\[ \underline\Omega_*(X) = \underline\cZ_*(X)/\langle\cR_*^{\text{Sect}}\rangle(X).\]

\item Let $\langle\LL_* \cR_*^{\text{FGL}}\rangle(X)$ be the $\LL_*$-submodule of $\LL_* \tensor \underline\Omega_*(X)$ generated by elements of the form
\[ [Y\to X, L_1,\ldots,L_r, L\tensor M] - \sum_{i,j\geq 0} a_{ij} [Y\to X, L_1,\ldots,L_r,  \underbrace{L,\ldots,L}_{i \textnormal{\ times}}, \underbrace{M,\ldots, M}_{j \textnormal{\ times}}],\]
where $a_{ij}\in \LL_*$ are the coefficients in the formal group law $F_\LL$. 
Define
\[ \Omega_*(X) = \LL_* \tensor \underline\Omega_*(X)/\langle\LL_* \cR_*^{\text{FGL}}\rangle(X).\]
\end{enumerate}

There is a natural homomorphism $\nu: \Mp(X) \to Z_*(X)$ mapping a cycle $[Y\to X] \in \Mp(X)$ to $[Y\to X] \in Z_*(X)$. It is proved in \cite{Levine-Pandharipande} that $\nu$ induces an isomorphism
\[ \nu: \omega_*(X)\to \Omega_*(X),\]
compatible with push-forward and pull-back homomorphisms and first Chern class operators.

Let us now define the theory $\hat\Omega_*(X)$. Let $\hat\cZ_*(X)$ be the graded free abelian group generated by the isomorphism classes of cobordism cycles 
\[ [f:Y\to X, L_1,\ldots,L_r],\]
where $Y$ is a smooth variety, $f$ is a proper morphism, and $L_i$ are line bundles on $Y$. We construct $\hat\Omega_*(X)$ by imposing relations on $\hat\cZ_*(X)$. Compared to the construction of $\Omega_*(X)$, only the first step needs to be changed to allow non-quasiprojective varieties $Z$.

\begin{enumerate}
 \item Let $\langle\hat\cR_*^{\text{Dim}}\rangle(X)$ be the subgroup of $\hat\cZ_*(X)$ generated by cobordism cycles of the form 
\[ [f:Y\to X, \pi^*(L_1),\ldots,\pi^*(L_r), M_1,\ldots, M_s],\]
where $Z$ is a smooth variety, $\pi: Y\to Z$ is a smooth morphism, $L_1,\ldots,L_r$ are line bundles on $Z$, and $r>\dim Z$. Define
\[ \underline{\hat\cZ}_*(X) = \hat\cZ_*(X)/\langle\hat\cR_*^{\text{Dim}}\rangle(X).\]

\item Let $\langle\hat\cR_*^{\text{Sect}}\rangle(X)$ be the subgroup of $\underline{\hat\cZ}_*(X)$ generated by differences of cobordism cycles of the form
\[ [Y\to X, L_1,\ldots,L_r] - [Z\to X, i^*(L_1),\ldots,i^*(L_{r-1})],\]
where $i: Z \to Y$ is the closed embedding of the zero locus of a transverse section of $L_r$. Define
\[ \underline{\hat\Omega}_*(X) = \underline{\hat\cZ}_*(X)/\langle\hat\cR_*^{\text{Sect}}\rangle(X).\]

\item Let $\langle\LL_* \hat\cR_*^{\text{FGL}}\rangle(X)$ be the $\LL_*$-submodule of $\LL_* \tensor \underline{\hat\Omega}_*(X)$ generated by elements of the form
\[ [Y\to X, L_1,\ldots,L_r, L\tensor M] - \sum_{i,j\geq 0} a_{ij} [Y\to X, L_1,\ldots,L_r, \underbrace{L,\ldots,L}_{i \textnormal{\ times}}, \underbrace{M,\ldots, M}_{j \textnormal{\ times}}],\]
where $a_{ij}\in \LL_*$ are the coefficients in the formal group law $F_\LL$.
Define
\[ \hat\Omega_*(X) = \LL_* \tensor \underline{\hat\Omega}_*(X)/\langle\LL_* \hat\cR_*^{\text{FGL}}\rangle(X).\]
\end{enumerate}
 
There is again a natural homomorphism $\hat\nu: \hat\cM^+(X)\to  \hat\cZ_*(X)$. We claim that it descends to a homomorphism 
\[ \hat\nu: \hat\omega_*(X)\to \hat\Omega_*(X).\]
This follows if we can prove that the double point relations in $\hat\cR(X)$ map to zero in $\hat\Omega_*(X)$. This statement for $\cR(X)$ and $\Omega_*(X)$ is proved in \cite{Levine-Pandharipande}(Corollary 10). The same proof works word by word in our case.

There is also an obvious homomorphism $\phi: \Omega_*(X)\to \hat\Omega_*(X)$, induced by the inclusion of cycles $\cZ_*(X)\hookrightarrow \hat\cZ_*(X)$. These homomorphisms give the following commutative square:
\[ \begin{CD}  
 \omega_*(X) @>{\psi}>>  \hat\omega_*(X)  \\
@V{\nu}VV @VV{\hat\nu}V\\
 \Omega_*(X) @>{\phi}>>  \hat\Omega_*(X).\\
\end{CD} \]
The maps in this square are compatible with proper push-forward homomorphisms, smooth quasiprojective pull-back homomorphisms and first Chern class operators.

\begin{theorem}\label{thm-main2} 
The homomorphism $\phi: \Omega_*(X) \to \hat\Omega_*(X)$ is an isomorphism for all $X$ in $\Schk$.
\end{theorem}

\begin{proof}
 We follow the same steps as in the proof of Theorem~\ref{thm-main}. We first show that $\phi$ is surjective, and then construct a left inverse $\delta$ of $\phi$, proving injectivity of $\phi$.

To prove surjectivity of $\phi$, it suffices to show that cobordism classes of the form $[Y\to X]\in \hat\cZ_*(X)$ lie in the image of $\phi$. These classes lie in the image of $\hat\nu$, hence they can be lifted to $\omega_*(X)$ and thus come from $\Omega_*(X)$. 

To prove injectivity of $\phi$, we construct a distinguished lifting $\delta: \hat\cZ_*(X) \to \Omega_*(X)$. Recall that if $[f:Y\to X]\in \hat{\cM}(X)$ then we constructed its distinguished lifting
\[ d [f:Y\to X] = f_*(d(1_Y)) \in \omega_*(X).\]
Define the distinguished lifting $\delta$ on generators
\[ \delta [f: Y\to X, L_1,\ldots,L_r] = f_* \ch(L_1)\circ \cdots \circ \ch(L_r)\circ \nu \circ d (1_Y).\]
Note that $f$ is in general only proper, hence we need proper push-forward in $\Omega_*$.
 When $Y$ is quasiprojective, then $d(1_Y) = 1_Y$, hence the composition
 \[ \cZ_*(X) {\longrightarrow} \hat{\cZ}_*(X)  \stackrel{\delta}{\longrightarrow} \Omega_*(X) \]
is the canonical homomorphism.

We need to show that $\delta$ descends to a homomorphism 
\[ \delta: \hat\Omega_*(X) \to \Omega_*(X).\]
We do this in three steps:

\begin{enumerate}
 \item Let $[f:Y\to X, g^*(L_1),\ldots,g^*(L_r), M_1,\ldots, M_s]$ be a generator of $\langle\hat\cR_*^{\text{Dim}}\rangle(X)$, where $g:Y\to Z$ is a smooth morphism between smooth varieties and $r>\dim Z$. We need to show that $\delta$ maps such cycles to zero, hence it induces a homomorphism $\delta: \underline{\hat\cZ}_*(X) \to \Omega_*(X)$. 

It suffices to prove that 
\[ \ch(g^*(L_1))\circ \cdots \circ \ch(g^*(L_r))\circ \nu \circ d (1_Y) = 0.\]
Let $\pi: W\to Y\times \PP^1$, $D=\pi^{*}(Y\times \{0\})$ define the distinguished lifting $d(1_Y)$. Similarly, let $\rho: U\to Z\times \PP^1$, $E=\rho^{*}(Z\times \{0\})$ define the distinguished lifting $d(1_Z)$. We may choose $\pi$ such that $W\to Y\times\PP^1\to Z\times\PP^1$ factors through $U$. Then every face $D^J$ of $D$ maps to some face $E^I$ or $E$. To apply $\ch(L_i)$ to $\nu\circ d(1_Y)$, we need to apply to the classes $1_{D^J}\in \Omega_*(D^J)$ the first Chern class operator of the pullback of $L_i$ along
\[ h: D^J \to |D| \to |E| \to Z.\]
Since this morphism factors through $D^J\to E^I$, where $E^I$ is quasiprojective and $\dim E^I \leq \dim Z < r$, it follows from the property (Dim) in $\Omega_*(D^J)$ that 
\[ \ch(h^* L_1)\circ \cdots \circ \ch(h^* L_r) (1_{D^J}) = 0.\]

\item Consider a generator 
\[ [Y\to X, L_1,\ldots,L_r] - [Z\to X, i^*(L_1),\ldots,i^*(L_{r-1})]\]
of $\langle\hat\cR_*^{\text{Sect}}\rangle(X)$. We need to prove that $\delta$ maps this element to zero. Lemma~\ref{lem-sect} proves that 
\[ \ch(L_r) \circ d (1_Y)  =d [Z\to Y],\]
hence the equality also holds after applying $f_* \circ \ch(L_1),\ldots, \ch(L_{r-1}) \circ \nu$.

\item We extend $\delta: \underline{\hat\Omega}_*(X) \to \Omega_*(X)$ $\LL_*$-linearly to a homomorphism
\[ \delta: \LL_*\tensor \underline{\hat\Omega}_*(X) \to \Omega_*(X).\]
We need to show that $\delta$ maps $\langle\LL_* \hat\cR_*^{\text{FGL}}\rangle(X)$ to zero. This is equivalent to the equality 
\[ \ch(L\tensor M)\circ \nu \circ d (1_Y) = \sum_{i,j} a_{i,j} \ch(L)^i \circ \ch(M)^j \circ \nu\circ d(1_Y)\]
for a smooth variety $Y$ and line bundles $L$ and $M$ on $Y$. The equality holds for any class $\alpha\in \Omega_*(Y)$ instead of $\nu \circ d (1_Y)$ by the property (FGL) in the theory $\Omega_*$.
\end{enumerate}
\end{proof}

We mentioned in the introduction that it is possible to strengthen the notion of an oriented Borel-Moore functor with products by requiring push-forward homomorphisms along proper morphisms and pull-back homomorphisms along smooth morphisms. Similarly, one can extend to this situation the notion of an oriented Borel-Moore functor of geometric type \cite{Levine-Morel}. The same argument as in \cite{Levine-Morel} then shows that $\hat\Omega_*(X)$ is a strengthened oriented Borel-Moore functor of geometric type and, moreover, that it is universal among all such functors.

\bibliographystyle{plain}
\bibliography{cobordismbib}

\begin{thebibliography}{1}

\bibitem{AKMW}
Dan Abramovich, Kalle Karu, Kenji Matsuki, and Jaroslaw Wlodarczyk.
\newblock Torification and factorization of birational maps.
\newblock {\em J. Amer. Math. Soc.}, 15(3):531--572 (electronic), 2002.

\bibitem{descentseq}
Jos\'e~Luis Gonz\'alez and Kalle Karu.
\newblock Descent for algebraic cobordism.
\newblock {\em arXiv: 1301.3292}.

\bibitem{Hironaka}
Heisuke Hironaka.
\newblock Resolution of singularities of an algebraic variety over a field of
  characteristic zero. {I}, {II}.
\newblock {\em Ann. of Math. (2) 79 (1964), 109--203; ibid. (2)}, 79:205--326,
  1964.

\bibitem{Levine-Morel}
Marc Levine and Fabien Morel.
\newblock {\em Algebraic cobordism}.
\newblock Springer Monographs in Mathematics. Springer, Berlin, 2007.

\bibitem{Levine-Pandharipande}
Marc Levine and Rahul Pandharipande.
\newblock Algebraic cobordism revisited.
\newblock {\em Invent. Math.}, 176(1):63--130, 2009.

\bibitem{Wlodarczyk}
Jaroslaw Wlodarczyk.
\newblock Toroidal varieties and the weak factorization theorem.
\newblock {\em Invent. Math.}, 154(2):223--331, 2003.

\end{thebibliography}

\end{document}